\documentclass[11pt,twoside]{article}
  \usepackage{geometry}
\usepackage{amssymb}
\usepackage{amsmath}
\usepackage{amsthm}
\usepackage{mathrsfs}
\usepackage{cite}
\usepackage{color}
\usepackage{changes}
\usepackage[colorlinks,
linkcolor=blue,
anchorcolor=blue,
citecolor=red
]{hyperref}

\def\zz{\mathbb{Z}}
\def\nn{\mathbb{N}}
\def\cp{\mathcal{P}}

\def\cc{\mathbb{C}}

\def\rn{\mathbb{R}^n}

\def\no{\nonumber}

\def\les{\lesssim}

\newtheorem{thm}{Theorem}[section]
\newtheorem{rem}[thm]{Remark}
\newtheorem{lem}[thm]{Lemma}

\newtheorem{defn}[thm]{Definition}

\def\XXint#1#2#3{{
\setbox0=\hbox{$#1{#2#3}{\int}$}
\vcenter{\hbox{$#2#3$}}\kern-.5\wd0}}

\allowdisplaybreaks[4]
\numberwithin{equation}{section}
\begin{document}
\title{ The Commutators of $n$-dimensional Rough Fractional Hardy Operators on Two Weighted Grand Herz-Morrey Spaces with Variable Exponents  \footnotetext{$\ast$ The corresponding author J. S. Xu  jingshixu@126.com\\
The work is supported by the National Natural Science Foundation of China (Grant No. 12161022) and the Science and Technology Project of Guangxi (Guike AD23023002)}}
\author{Shengrong Wang\textsuperscript{a}, Pengfei Guo\textsuperscript{a}, Jingshi Xu\textsuperscript{b,c,d*}\\
{\scriptsize  \textsuperscript{a} School of Mathematics and Statistics, Hainan Normal University, Haikou, 571158, China }\\
{\scriptsize  \textsuperscript{b}School of Mathematics and Computing Science, Guilin University of Electronic Technology, Guilin 541004, China} \\
{\scriptsize  \textsuperscript{c} Center for Applied Mathematics of Guangxi (GUET), Guilin 541004, China}\\
{\scriptsize  \textsuperscript{d}Guangxi Colleges and Universities Key Laboratory of Data Analysis and Computation, Guilin 541004, China}}
\date{}

\maketitle

{\bf Abstract.} In this paper, we  obtain the boundedness of $m$th order commutators generated by the $n$-dimensional fractional Hardy operator with rough kernel and its adjoint operator with BMO functions  on two weighted grand Herz-Morrey spaces with variable exponents.
Replacing  Lipschitz functions with BMO functions the corresponding  result is also given.

 

\section{Introduction}
Denote by $L^{1}_{\rm loc}(\rn)$ the set of all complex-valued locally integrable functions on $\rn.$ 
Let $0\leq \beta<n$ and $f \in  L^{1}_{\rm loc}(\rn)$. In \cite{fllw},  Fu, Liu and Lu  defined the $n$-dimensional fractional Hardy operator and its adjoint operator as
\[H_\beta f(x):= \frac{1}{|x|^{n-\beta}}\int_{|t| \leq |x|} f(t){\rm d} t,\ x \in \mathbb{R}^n \backslash \{ 0\},\]
and 
\[H^\ast_\beta f(x):= \frac{1}{|x|^{n-\beta}}\int_{|t| > |x|} f(t){\rm d} t,\ x \in \mathbb{R}^n \backslash \{ 0\}. \]
When $\beta=0$, $H_0$ and $H^\ast_0$ are the $n$-dimensional Hardy operator and its adjoint operator introduced by Christ and Grafakos, respectively. They are bounded on $L^p(\rn)$ for $p \in (1,\infty)$, see \cite{7}. The boundedness of commutators generated by the fractional Hardy operator and its adjoint operator with CMO$(\rn)$ functions on products of Lebesgue spaces and homogeneous Herz spaces was given in \cite{fllw}.

Let $\mathbb{S}^{n-1}$ be the unit sphere in $\rn$ equipped with the Lebesgue measure ${\rm d} \sigma = {\rm d} \sigma(x^{\prime})$. Let  $\Omega$ be a homogeneous function of degree zero and satisfy
\[ \int_{\mathbb{S}^{n-1}} \Omega(x^{\prime}) {\rm d} \sigma(x^{\prime})=0, \]
where $x^{\prime}=x/|x|$ and $x$ is not zero.
In \cite{flz} , Fu, Lu and Zhao  proposed the $n$-dimensional Hardy operator with a rough kernel and its adjoint operator are defined as
\[H_{\Omega,\beta}: = \frac{1}{|x|^{n-\beta}} \int_{|t|\leq |x|} \Omega(x-t) f(t) {\rm d} t, \quad x\in \rn \setminus \{0\}\]
and
\[H^\ast_{\Omega,\beta} :=  \int_{|t| > |x|} \frac{\Omega(x-t) f(t)}{|t|^{n-\beta}} {\rm d} t, \quad x\in \rn \setminus \{0\},\]
where $\Omega \in L^s(\mathbb{S}^{n-1})$, $1 \leq s < \infty$.

Let $b \in L^{1}_{\rm loc}(\rn)$. 
Then the $m$th order commutators generated by the Hardy operator with rough kernel and its adjoint  operator with $b$  are defined as 
\[H^{b,m}_{\Omega,\beta} := \frac{1}{|x|^{n-\beta}} \int_{|t|\leq |x|} (b(x)-b(t))^m \Omega(x-t) f(t) {\rm d} t, \quad x\in \rn \setminus \{0\}\]
and
\[H^{\ast,b,m}_{\Omega,\beta} := \frac{1}{|x|^{n-\beta}} \int_{|t| > |x|} \frac{ (b(x)-b(t))^m \Omega(x-t) f(t)}{|t|^{n-\beta}} {\rm d} t, \quad x\in \rn \setminus \{0\}\]
for suitable functions $f$. Then they obtained the boundedness of $H^{b,m}_{\Omega,\beta}$ and $H^{\ast,b,m}_{\Omega,\beta}$ with $b$ being central BMO functions on Lesbegue and Herz spaces. Furthermore, $\lambda$-central BMO estimates for commutators on central Morrey spaces were also discussed in \cite{flz}.

Let  $0\leq \beta<n$. The fractional integral operator $I_\beta$ is defined as
\[I_\beta (f)(x):= \int_{\rn} \frac{f(y)}{|x-y|^{n-\beta}} {\rm d} y, \ x \in \rn. \]

Izuki and  Noi \cite{in-1} proved that the fractional integral operator from weighted Herz spaces  with variable exponent $\dot{K}^{\alpha,q_1}_{p_1 (\cdot)}(w)$  into $\dot{K}^{\alpha,q_2}_{p_2 (\cdot)}(w)$ for $0<q_1\leq q_2 <\infty$, $0<\beta<n/p^+_1$ and $1/p_2(\cdot)=1/p_1(\cdot)-\beta/n$.
Izuki \cite{im-1} showed the boundedness of fractional integrals from the Herz-Morrey space with variable exponent $M\dot{K}^{\alpha,\lambda}_{q_1,p_1 (\cdot)}(\rn)$ to $M\dot{K}^{\alpha,\lambda}_{q_2,p_2 (\cdot)}(\rn)$ for $0<\lambda<\alpha$, $0<q_1\leq q_2 <\infty$, $0<\beta<n/p^+_1$ and $1/p_2(\cdot)=1/p_1(\cdot)-\beta/n$.
Liu, Zhang and Yao \cite{lzy-1} obtained the boundedness for $m$-th order commutators of the $n$-dimensional fractional Hardy operator and adjoint operator on weighted variable exponent Morrey-Herz space $M\dot{K}^{\alpha(\cdot),\lambda}_{q,p (\cdot)}(w)$.
Niu and Wang \cite{nw-1} gave the boundedness of Hardy-type operator with rough kernels and their commutators on central Morrey space with variable exponents $\dot{\mathcal{B}}^{p(\cdot),\lambda}$.
Many results of the boundedness of these operators can be seen in \cite{im-2,wz-1,in-2,sb1,zs1,sa1,wj1,ahs1,xy1}

Recently, Nafis, Rafeiro and Zaighum \cite{nrz-1,nrz-2,nrz-3} introduced the grand variable Herz space, and obtained the boundedness of some operators on these spaces.
Zhang, He and Zhang \cite{zhz2} introduced weighted grand Herz-Morrey type spaces  and showed the boundedness of sublinear operators and their multilinear commutators on these spaces.
Sultan and Sultan \cite{sbm1} obtained the boundedness of the higher order commutators of the Hardy operators on grand variable Herz-Morrey spaces.
 Chen, Lu and Tao \cite{clt1} established the boundedness of fractional Marcinkiewicz integral  and its higher order commutator  is bounded on grand variable Herz spaces  and grand variable Herz-Morrey spaces.

Motivated by the mentioned works, in this paper we consider the boundedness of the $n$-dimensional rough fractional Hardy operator on two weighted grand Herz-Morrey spaces with variable exponents.
 The plan of the paper is as follows. In Section \ref{wbl-s-1}, we collect some notations.  
In Section \ref{wbl-s-2}, we  give the boundedness of $m$th order commutators generated by the $n$-dimensional fractional Hardy operator with rough kernel and its adjoint operator with BMO function  on two weighted grand Herz- Morrey spaces with variable exponents.
In Section \ref{wbl-s-3}, we show that the boundedness of $m$th order commutators generated by the $n$-dimensional fractional Hardy operator with rough kernel and its adjoint operator with Lipschitz functions  on two weighted grand Herz- Morrey spaces with variable exponents. 

\section{Notations and preliminaries}\label{wbl-s-1}

In this section, we first recall some definitions and notations. Let $\nn$ be the collection of all natural numbers and $\nn_0 = \nn \cup \{0\}$. Let $\zz$ be the collection of all integers.
Let $\rn$ be the $n$-dimensional Euclidean space, where $n \in \nn$. In the sequel, $C$ denotes positive constants, but it may change from line to line. For any quantities $A$ and $B$, if there exists a constant $C>0$ such that $A\leq CB$, we write $A \lesssim B$. If $A\lesssim B$ and $B \lesssim A$, we write $A \sim B$.
For each $k \in \mathbb{Z}$, we define
${B_k}: = \{ x \in \rn:| x | \leq {2^k} \},$ $D_k: = B_k\backslash B_{k - 1}=\{ x\in \rn: 2^{k-1}<|x| \leq 2^k\},$ $\chi _k: = \chi _{D_k},$ $\widetilde{\chi}_m=\chi_{m},$
$m \geq 1,$ $\widetilde{\chi}_0=\chi_{B_0}.$

 For a measurable function $q(\cdot)$ on $\rn$ taking values in $\ (0,\infty)$, we denote ${q^- }: =  {\rm ess}\inf_{x \in \rn} q(x),$ ${q^+ }: =  {\rm ess} \sup_{x \in\rn} q(x).$
The set $\cp(\rn)$ consists of all $q(\cdot)$ satisfying $q^->1$ and $q^+<\infty;$
$\cp_{0}(\rn)$ consists of all $q(\cdot)$ satisfying $q^->0$ and $q^+<\infty$.

Let $q(\cdot) \in \cp_0(\rn)$.
Then the Lebesgue space with variable exponent $L^{q(\cdot)}(\rn)$ is defined by
\[L^{q(\cdot)}(\rn): = \bigg\{f \text{ is measurable on}\ \rn:\ \int_{\rn} \bigg( \frac{|f(x)|}
{\lambda } \bigg)^{q(x)}\, {\rm d} x < \infty\text{ for some }\lambda>0 \bigg\}. \]
The space $L^{q(\cdot)}(\rn)$ is  equipped with the (quasi) norm
\[\| f \|_{L^{q(\cdot)}}: = \inf \bigg\{ \lambda  > 0:\int_{\rn} \bigg( \frac{|f(x)|}{\lambda } \bigg)^{q(x)}{\rm d} x \leq 1  \bigg\}.\]
Denote by $L^\infty(\rn)$ the set of all measurable functions $f$ such that
\[ \|f\|_{L^\infty} := \mathop{{\rm ess}\sup}_{y \in \rn} |f(y)| < \infty.\]
The space $L^{q(\cdot)} _{\rm loc}(\rn)$ is defined by
\[L^{q(\cdot)} _{\rm loc}(\rn):=\{f:f\chi_{K} \in L^{q(\cdot)}(\rn) {\text{ for all compact subsets }} K \subset \rn\},\]
where and what follows, $\chi_{S}$ denotes the characteristic function of a measurable set $S\subset \rn.$

Let $f\in L^{1}_{\rm loc}(\rn)$. Then the standard Hardy-Littlewood maximal function of $f$ is defined by
\[\mathcal{M}f(x) := \sup_{B \ni x} \frac{1}{| B |}\int_B |f(y)|dy , \ \forall x \in \mathbb{R}^n,\]
where the supremum is taken over all balls containing $x$ in $\rn$. 

\begin{defn}\label{cd-1}
Let $\alpha(\cdot)$ be a real-valued measurable function on $\rn$.

{\rm (i)} The function $\alpha(\cdot)$ is locally $\log$-H\"{o}lder continuous if there exists a constant $C_{\log}(\alpha)$ such that
\[|\alpha (x) - \alpha (y)| \leq \frac{C_{\log}(\alpha)}{\log ( e + 1/|x - y|)}, \ x,y \in \mathbb{R}^n,\ |x - y| < \frac{1}{2}.\]
Denote by $C_{\rm loc}^{\log}(\rn)$ the set of all locally log-H\"{o}lder continuous

{\rm (ii)} The function $\alpha(\cdot)$ is $\log$-H\"{o}lder continuous at the origin if there exists a constant $C_0$ such that
\[|\alpha (x) - \alpha (0)| \leq \frac{C_0} {\log ( e + 1/| x | )}, \ \forall x \in \rn.\]
Denote by $C^{\log}_0(\rn)$ the set of all $\log$-H\"{o}lder continuous functions at the origin.

{\rm (iii)} The function $\alpha(\cdot)$ is $\log$-H\"{o}lder continuous at infinity if there exists $\alpha_{\infty}\in \mathbb{R}$ and a constant $C_\infty$ such that
\[|\alpha (x) - \alpha _\infty|\leq \frac{C_\infty} {\log ( e + | x | )}, \ \forall x \in \rn.\]
Denote by $C^{\log}_\infty(\rn)$ the set of all $\log$-H\"{o}lder continuous functions at infinity.

{\rm (iv)} The function $\alpha(\cdot)$ is global $\log$-H\"{o}lder continuous if $\alpha(\cdot)$ are both locally $\log$-H\"{o}lder continuous and $\log$-H\"{o}lder continuous at infinity.
Denote by $C^{\log}(\rn)$ the set of all global $\log$-H\"{o}lder continuous functions.
\end{defn}

Let $q(\cdot) \in \mathcal{P}(\rn)$ and $w$ be  a nonnegative measurable function on $\mathbb{R}^n$.
Then the weighted variable exponent Lebesgue space $L^{q(\cdot)}(w)$ is  the set of all complex-valued measurable functions $f$ such that $fw \in L^{q(\cdot)}$. The space $L^{q(\cdot)}(w)$ is a Banach space equipped with the norm
\[\|f\|_{L^{q(\cdot)}(w)}:=\|fw\|_{L^{q(\cdot)}}.\]

For weights, a important class is the following Muckenhoupt $A_q$ class with constant exponent $q \in [1,\infty]$ which was firstly proposed by Muckenhoupt in \cite{29}.

\begin{defn}\label{vd3}
Fix $q \in (1,\infty)$. A positive measurable function $w$ is said to be in the Muckenhoupt class $A_q$, if 
\[ [W]_{A_q}:=\sup_{\text{all ball} \ B \subset \rn} \bigg(\frac{1}{|B|}\int_B w(x) {\rm d}x \bigg) \bigg(\frac{1}{|B|}\int_B w(x)^{1-q^{\prime}}{\rm d}x \bigg)^{q-1}<\infty.\]
We say $w \in A_1$, if there exists a positive constant $C$ such that $\mathcal{M}w(x) \leq Cw(x)$ for a.e. $x\in \rn$. We denote $A_\infty :=\cup_{p\geq1} A_q$.
\end{defn}

In \cite{bt23}, Cruz-Uribe, Fiorenza and Neugebauer generalized the Muckenhoupt $A_q$ class with constant exponent  to variable exponent.
\begin{defn}\label{dg1}
Let $q(\cdot) \in \mathcal{P}(\rn)$, a nonnegative measurable function $w$ is said to be in $A_{q(\cdot)}$,
if 
\[ \|w\|_{A_{q(\cdot)}}:= \sup_{\text{all ball } B \subset \rn} \frac{1}{|B|}\|w\chi_B\|_{L^{q(\cdot)}} \|w^{-1}\chi_B\|_{L^{q^{\prime}(\cdot)}} < \infty.\]
\end{defn}

\begin{defn}\label{dg2}
Let $q(\cdot) \in \mathcal{P}(\mathbb{R}^n)$, a nonnegative measurable function $w$ is said to be in $\tilde{A}_{q(\cdot)}$,
if 
\[\|w\|_{\tilde{A}_{q(\cdot)}}= \sup_{\text{all ball } B \subset \rn} \frac{1}{|B|}\|w^{1/q(\cdot)}\chi_B\|_{L^{q(\cdot)}} \|w^{-1/q(\cdot)}\chi_B\|_{L^{q^{\prime}(\cdot)}}  \leq \infty.\]
\end{defn}

To describe the boundedness of the fractional integral on the weighted function space, Bernardis Dalmasso and  Pradolini \cite{bdp1} introduced a class weight $A_{p(\cdot),q(\cdot)}$ related to it.
\begin{defn}
Let $p(\cdot)$, $q(\cdot)\in \cp(\rn)$ and $0<\beta<n$ such that $1/p(x)-1/q(x)=\beta/n$. We say that $w \in A_{p(\cdot),q(\cdot)}$ if there exist a positive constant $C$ for all ball $B$ in $ \rn$ such that
\[ \|w\chi_B\|_{L^{q(\cdot)}} \|w^{-1}\chi_B\|_{L^{p^{\prime}(\cdot)}} \leq C|B|^{1-\frac{\beta}{n}}. \]
\end{defn}

Sultan et al. \cite{sbm2} introduced the weighted grand Herz space $\dot{K}_{q(\cdot)}^{\alpha,p),\theta}(w)$ and they obtained the boundedness of fractional integrals.
In \cite{in1}, Izuki and  Noi introduced two weighted Herz spaces $\dot{K}_{p(\cdot)}^{\alpha (\cdot),q(\cdot)}(w_1,w_2)$ with variable exponents.  Based on the above, we consider the two weight grand Herz-Morrey space with variable exponents. 

\begin{defn}  Let $v$, $w$ be weights on $\rn$, $q(\cdot)\in\cp(\rn),$ $p\in (1,\infty)$, $\lambda \in [0, \infty)$, $\alpha(\cdot) \in L^\infty(\rn)$ and $\theta\in (0,\infty)$.
  The homogeneous two weight grand Herz-Morrey space $M\dot{K}_{q(\cdot),\lambda}^{\alpha (\cdot),p),\theta}(v,w)$ is defined  by
\[M\dot{K}_{q(\cdot),\lambda}^{\alpha (\cdot),p),\theta}(v,w):= \Big\{f\in L^{p(\cdot)}_{\rm loc}(\rn \setminus \{0\},w):\|f\|_{M\dot{K}_{q(\cdot),\lambda}^{\alpha (\cdot),p),\theta}(v,w)}<\infty\Big\},\]
where
\[\|f\|_{M\dot{K}^{\alpha(\cdot),p),\theta}_{q(\cdot),\lambda}(v,w)}:= \sup_{\delta>0} \sup_{k_0\in \zz}2^{-k_0\lambda} \bigg( \delta^\theta \sum_{k=-\infty}^{k_0} \big\| v(B_k)^{\alpha(\cdot)/n} f\chi_k \big\|_{L^{q(\cdot)}(w)}^{p(1+\delta)} \bigg)^{\frac{1}{p(1+\delta)}}. \]
\end{defn}
The following lemma   is well known, for example, see \cite{ta1,dj1,gj1}.
\begin{lem}\label{tL3}
Let $w$ be a weight on $\rn$. If $q \in [1,\infty)$ and $w \in A_q$, then there exist constants $\epsilon\in (0,1)$ and $C>0$ such that for all balls $B$ in $\rn$ and all measurable subsets $S\subset B,$
\begin{equation*}
\frac{w(B)}{w(S)} \leq C \bigg( \frac{|B|}{|S|} \bigg)^q,
\end{equation*}
\begin{equation*}
\frac{w(S)}{w(B)} \leq C \bigg( \frac{|S|}{|B|} \bigg)^\epsilon.
\end{equation*}
\end{lem}

The following lemma  has been proved by Noi and Izuki in \cite{in-1}.
\begin{lem}\label{tt-L5}
If $q(\cdot)\in C^{\log}(\rn)\cap \mathcal{P}(\rn)$ and $w \in A_{q(\cdot)}$,
then there exist constants $\delta_{1}$, $\delta_{2}\in (0,1)$ and $C>0$ such that for all balls $B$ in $\rn$ and all measurable subsets $S\subset B,$
\begin{equation*}
\frac{\|\chi _S\|_{L^{q( \cdot )}(w)}}{\| \chi _B \|_{L^{q( \cdot )}(w)}} \leq C\bigg( \frac{| S |}{| B |} \bigg)^{\delta _1},
\end{equation*}
\begin{equation*}
\frac{\|\chi _S\|_{L^{q^{\prime}( \cdot )}(w^{-1})}}{\|\chi _B \|_{L^{q^{\prime}( \cdot )}(w^{-1})}} \leq C\bigg( \frac{| S |}{| B |} \bigg)^{\delta _2}.
\end{equation*}
\end{lem}

\begin{rem}
 If $w^{q(\cdot)} \in \tilde{A}_{q(\cdot)}$, then $w^{-q^{\prime}(\cdot)} \in \tilde{A}_{p^{\prime}(\cdot)}$  by Definitions \ref{dg2}. Then if  $w^{q(\cdot)} \in \tilde{A}_{q(\cdot)}$, we have $w^{q(\cdot)} \in A_{q(\cdot)}$ and $w^{-q^{\prime}(\cdot)} \in A_{q^{\prime}(\cdot)}$ by Definitions \ref{dg1}. 
 \end{rem}

\begin{lem}[see {\cite[Corollary 4.5.9]{dhhr-0}}]\label{tt-L4}
Let $q(\cdot) \in \mathcal{P}(\rn)\cap C^{\log}(\rn)$. Then
\begin{equation*}
\|\chi_Q\|_{L^{q(\cdot)}} \sim
\begin{cases}
|Q|^{1/q(x)} & \text{if } |Q| \leq 2^n \text{ and } x \in Q,  \\
|Q|^{1/q_\infty} & \text{if } |Q| \geq 1,
\end{cases}
\end{equation*}
for every cube (or ball) $Q \subset \rn$, where $q_\infty = \lim_{x\rightarrow \infty}q(x)$.
\end{lem}

The following lemma  comes from \cite{flz}.
\begin{lem}\label{tt-L2}
Let $\Omega \in L^s(\mathbb{S}^{n-1})$. Then there exists a positive constant C such that for  $x \in D_k$, $j \leq k$
\[ \int_{D_j} |\Omega(x-t)|^s {\rm d} t \leq \int_0^{2^{k+1}} \int_{\mathbb{S}^{n-1}} |\Omega(x')|^s d\sigma(x') r^{n-1} {\rm d} r \leq C2^{kn}.\]
\end{lem}

\begin{lem}[see {\cite[Lemma 7]{in1}}]\label{tL2}
Let $k,l \in \zz,$ $w \in A_q$ with $q \in [1,\infty),$ $\epsilon \in (0,1)$ be the constant in Lemma \ref{tL3}, $w^-=
\begin{cases}
\epsilon \quad \text{if} \quad \alpha^-\geq 0\\
 q \quad \text{if} \quad \alpha^-< 0
\end{cases}$
and $w^+=
\begin{cases}
q \quad \text{if} \quad \alpha^+ \geq 0\\
\epsilon \quad \text{if} \quad \alpha^+ < 0
\end{cases}.$
If $\alpha(\cdot) \in L^\infty(\rn)$ is $\log$-H\"older continuous both at the origin and infinity, then there is a positive constant $C$ such that for any $x \in D_k$ and $y \in D_l,$
\[ w(B_k)^{\alpha(x)} \leq C w(B_l)^{\alpha(y)} \times \begin{cases}
2^{(k-l)nw^-\alpha^-} \quad &\text{if} \quad l \leq k-1\\
1    &            \text{if} \quad k-1< l \leq k+1\\
 2^{(k-l)nw^-\alpha^-} \quad &\text{if} \quad l > k+1
\end{cases}.\]
\end{lem}

\begin{lem}[see {\cite[Lemma 8]{in1}}]\label{tL4}
If $\alpha(\cdot) \in L^\infty(\rn)$ and is $\log$-H\"older continuous both at the origin and infinity, then for all $k \in \zz$ and $x \in D_k$,
\begin{align*}
w(D_k)^{\alpha(x)} \approx w(D_k)^{\alpha_\infty},  \ \text{if } k \geq 0, \quad\\
w(D_k)^{\alpha(x)} \approx w(D_k)^{\alpha(0)},  \ \text{if } k \leq -1.
\end{align*}
\end{lem}

\begin{lem}[see {\cite[Corollary 3.8]{cw-1}}]\label{tt-L3}
Let $q_2(\cdot)\in C^{\log}(\rn) \cap \mathcal{P}(\rn)$, $0<\beta<n/q^+_1$ and $q_1(\cdot)$ satisfying $1/q_2(\cdot)=1/q_1(\cdot) - \beta/n$. 
If $s \in A_{q_1(\cdot),p_2(\cdot)}$, then $I_\beta$ is bounded from $L^{q_1(\cdot)}(w)$ to $L^{q_2(\cdot)}(w)$.
\end{lem}

\section{Commutators of the Hardy operator and BMO functions }\label{wbl-s-2}

In this section, we show that the boundedness of $m$th order commutators generated by the $n$-dimensional fractional Hardy operator with rough kernel and its adjoint operator with BMO function  on two weighted grand Herz-Morrey spaces with variable exponents.

\begin{lem}[see {\cite[Corollary 3.11]{wx2}}]\label{L10}
Let $b\in {\rm{BMO}}(\mathbb{R}^{n})$, $q(\cdot)\in \mathcal{P}^{\log}(\rn)\cap \mathcal{P}(\rn)$, $w \in A_{q(\cdot)}$, $m\in [1,\infty)$. Then there exists a constant $C>0$ for  $k,i\in \mathbb{N}$ such that $k>i$ 
\[  \| |b - b_{B_i}|^m \chi _{B_k} \|_{L^{q( \cdot )}(w)} \leq C(k - i)^m \| b \|^m_\ast \| \chi _{B_k} \|_{L^{q( \cdot )}(w)}.\]
\end{lem}

\begin{thm}\label{bw-T1}
Let $0<\beta<n/q^+_1$,  $\theta \in (0,\infty)$, $\lambda \in [0, \infty)$, $b \in {\rm BMO}(\rn)$, $\alpha(\cdot) \in L^\infty(\rn)$,  $1<p<\infty$,  $q_2(\cdot)\in C^{\log}(\rn) \cap \cp(\rn)$ and $q_1(\cdot)$ satisfying $1/q_2(\cdot)=1/q_1(\cdot) - \beta/n$.    
Let $v\in A_r$ for some $r \in [1,\infty)$, $w \in A_{q_2(\cdot)}$, $\Omega \in L^s(S^{n-1})$ with $s>q^{\prime}_1(\cdot)$ and $w^+\alpha^+ + \beta+n/s-n\delta_2<0$,  where  $\delta_2 \in (0,1)$ is the constant in Lemma \ref{tt-L5} for $q(\cdot)$. 
Suppose that  $\alpha(\cdot)$ is log-H\"older continuous at infinity and at the origin. Then
\[\|\mathcal{H}^{b,m}_{\Omega,\beta} f\|_{M\dot{K}_{q_2(\cdot),\lambda}^{\alpha (\cdot),p),\theta}(v,w)} \les  \|b\|^m_\ast \| f\|_{M\dot{K}_{q_1(\cdot),\lambda}^{\alpha (\cdot),p),\theta}(v,w)}.\]
\end{thm}

\begin{proof}
 Since bounded functions with compact support are dense in $M\dot{K}_{q_2(\cdot),\lambda}^{\alpha (\cdot),p),\theta}(v,w)$, we only consider that $f$ is a bounded functions with compact support and write
\[f=\sum^\infty_{j=-\infty}f \chi_j=:\sum^\infty_{j=-\infty}f_j.\]
By Lemma \ref{tL2}, we have
\begin{align*}
&|v(B_k)^{\alpha(x)/n} \mathcal{H}^{b,m}_{\Omega,\beta} f(x) \chi_k(x)| \\
&\leq  v(B_k)^{\alpha(x)/n} \frac{1}{|x|^{n-\beta}} \int_{B_k} |b(x)-b(t)|^m |\Omega(x-t) f(t)| {\rm d} t \chi_k(x)\\
&\leq  v(B_k)^{\alpha(x)/n} \frac{1}{|x|^{n-\beta}} \sum_{j=-\infty}^k \int_{D_j} |b(x)-b(t)|^m |\Omega(x-t) f(t)| {\rm d} t \chi_k(x)\\
& \leq 2^{k(\beta-n)} \sum_{j=-\infty}^k 2^{(k-j)w^+\alpha^+}  \int_{D_j} |b(x)-b(t)|^m | v(B_j)^{\alpha(t)/n} \Omega(x-t) f(t)| {\rm d} t \chi_k(x)\\
& \leq 2^{k(\beta-n)} \sum_{j=-\infty}^k 2^{(k-j)w^+\alpha^+}  \int_{D_j} |b(x)-b_{D_j}|^m | v(B_j)^{\alpha(t)/n} \Omega(x-t) f(t)| {\rm d} t \chi_k(x)\\
& \quad + 2^{k(\beta-n)} \sum_{j=-\infty}^k 2^{(k-j)w^+\alpha^+}  \int_{D_j} |b(t)-b_{D_j}|^m | v(B_j)^{\alpha(t)/n} \Omega(x-t) f(t)| {\rm d} t \chi_k(x)\\
&=: E_1 + E_2.
\end{align*}
Estimate $E_1$. By  H\"{o}lder's inequality, we have
\begin{align}\label{ow-1}
E_1 &\les 2^{k(\beta-n)} \sum_{j=-\infty}^k 2^{(k-j)w^+\alpha^+} |b(x)-b_{D_j}|^m \no \\
& \quad \times \|v(B_j)^{\alpha(\cdot)/n} f_j\|_{L^{q_1(\cdot)}(w)} \|\Omega(x-\cdot)\chi_j\|_{L^{q'_1(\cdot)}(w^{-1})}  \chi_k(x) .
\end{align}
Taking the $L^{q_2(\cdot)}(w)$-norm on both sides of (\ref{ow-1}), then by Lemma \ref{L10}, we have 
\begin{align}\label{ow-2}
\|E_1 \|_{L^{q_2(\cdot)}(w)}
& \les 2^{k(\beta-n)} \sum_{j=-\infty}^k 2^{(k-j)w^+\alpha^+} \| |b(x)-b_{D_j}|^m \chi_k\|_{L^{q_2(\cdot)}(w)}  \no\\
& \quad \times \|v(B_j)^{\alpha(\cdot)/n} f_j\|_{L^{q_1(\cdot)}(w)}  \|\Omega(x-\cdot)\chi_j\|_{L^{q'_1(\cdot)}(w^{-1})}  \no \\
& \les 2^{k(\beta-n)} \sum_{j=-\infty}^k 2^{(k-j)w^+\alpha^+}   \|b\|^m_\ast  (k-j)^m \|v(B_j)^{\alpha(\cdot)/n} f_j\|_{L^{q_1(\cdot)}(w)} \no \\
& \quad \times   \|\Omega(x-\cdot)\chi_j\|_{L^{q'_1(\cdot)}(w^{-1})} \| \chi_k\|_{L^{q_2(\cdot)}(w)}.
\end{align}
Estimate $E_2$. By  H\"{o}lder's inequality again, we have
\begin{align}\label{ow-3}
E_2 &\les   2^{k(\beta-n)} \sum_{j=-\infty}^k 2^{(k-j)w^+\alpha^+} \|v(B_j)^{\alpha(\cdot)/n} f_j\|_{L^{q_1(\cdot)}(w)} \no \\
& \quad \times  \| |b(x)-b_{D_j}|^m \Omega(x-\cdot)\chi_j\|_{L^{q'_1(\cdot)}(w^{-1})}  \chi_k(x) .
\end{align}
Taking the $L^{q_2(\cdot)}(w)$-norm on both sides of (\ref{ow-3}), then by Lemma \ref{L10}, we have
\begin{align}\label{ow-4}
\|E_2 \|_{L^{q_2(\cdot)}(w)} 
&\les  2^{k(\beta-n)} \sum_{j=-\infty}^k 2^{(k-j)w^+\alpha^+}  \|v(B_j)^{\alpha(\cdot)/n} f_j\|_{L^{q_1(\cdot)}(w)} \no \\
& \quad \times  \| |b(x)-b_{D_j}|^m \Omega(x-\cdot)\chi_j\|_{L^{q'_1(\cdot)}(w^{-1})}  \|\chi_k(x) \|_{L^{q_2(\cdot)}(w)} \no\\
& \les   2^{k(\beta-n)} \|b\|^m_\ast \sum_{j=-\infty}^k 2^{(k-j)w^+\alpha^+} \|v(B_j)^{\alpha(\cdot)/n} f_j\|_{L^{q_1(\cdot)}(w)} \no\\
& \quad \times  \| \Omega(x-\cdot)\chi_j\|_{L^{q'_1(\cdot)}(w^{-1})}  \|\chi_k(x) \|_{L^{q_2(\cdot)}(w)}.
\end{align}
Thus, by (\ref{ow-2}) and (\ref{ow-4}), we obtain
\begin{align}\label{ow-5}
&\|v(B_k)^{\alpha(\cdot)/n} \mathcal{H}^{b,m}_{\Omega,\beta} (f) \chi_k \|_{L^{q_2(\cdot)}(w)} \no \\
&\quad  \les   2^{k(\beta-n)} \|b\|_\ast^m  \sum_{j-\infty}^k (k-j)^m 2^{(k-j)w^+\alpha^+} \|v(B_j)^{\alpha(\cdot)/n} f_j\|_{L^{q_1(\cdot)}(w)} \no \\
& \quad \quad \times  \| \Omega(x-\cdot)\chi_j\|_{L^{q'_1(\cdot)}(w^{-1})}  \|\chi_k(x) \|_{L^{q_2(\cdot)}(w)}.
\end{align}
Since $s>q_1^{\prime}(\cdot)$, there exists $\gamma^{\prime}(\cdot)$ such that $1/q^{\prime}_1(\cdot)=1/s + 1/\gamma^{\prime}(\cdot)$. Then by  H\"{o}lder's inequality and Lemma \ref{tt-L2}, we obtain
\begin{equation}\label{ow-6}
\|\Omega(x-\cdot)\chi_j\|_{L^{q'_1(\cdot)}(w^{-1})} \les \|\Omega(x-\cdot)\|_{L^s} \|\chi_j\|_{L^{\gamma'(\cdot)}(w^{-1})} \les 2^{kn/s} \|\chi_j\|_{L^{\gamma'(\cdot)}(w^{-1})}.
\end{equation}
Note that $\|\chi_j\|_{L^{q_2(\cdot)}(w)} \leq \|\chi_{B_j}\|_{L^{q_2(\cdot)}(w)}$ and $\chi_{B_j} \les 2^{-j \beta} I_\beta (\chi_{B_j})$.  Thus, by Lemma \ref{tt-L3} and Definition \ref{dg1}, we obtain
\begin{align}\label{ow-7}
\|\chi_j\|_{L^{q_2(\cdot)}(w)}
& \les 2^{-j\beta} \|I_\beta (\chi_{B_j})\|_{L^{q_2(\cdot)}(w)} \les 2^{-j(\beta+\theta)} \|\chi_{B_j}\|_{L^{q_1(\cdot)}(w)} \no \\
&\les 2^{-j\beta} \|\chi_j\|_{L^{q_1(\cdot)}(w)}  \les 2^{j(n-\beta)} \|\chi_j\|^{-1}_{L^{q'_1(\cdot)}(w^{-1})}.
\end{align}
By Lemmas \ref{tt-L4} and \ref{tt-L5}, Definition \ref{dg1}, (\ref{ow-6}) and (\ref{ow-7}), we have
\begin{align}\label{ow-8}
& 2^{(k-j)w^+\alpha^+} 2^{k(\beta-n+n/s)} \|\chi_j\|_{L^{\gamma'(\cdot)}(w^{-1})} \|\chi_k\|_{L^{q_2(\cdot)}(w)} \no\\
& \quad \les 2^{(k-j)w^+\alpha^+} 2^{k(\beta-n+n/s)} 2^{-jn/s} \|\chi_j\|_{L^{q'_1(\cdot)}(w^{-1})} \|\chi_k\|_{L^{q_2(\cdot)}(w)} \no\\
&\quad \les 2^{(k-j)w^+\alpha^+} 2^{k(\beta+n/s)} 2^{-jn/s} \|\chi_j\|_{L^{q'_1(\cdot)}(w^{-1})} \|\chi_k\|^{-1}_{L^{q'_2(\cdot)}(w^{-1})} \no\\
&\quad \les  2^{(k-j)w^+\alpha^+} 2^{k(\beta+n/s)} 2^{n\delta_2(j-k)} 2^{-jn/s} \|\chi_j\|_{L^{q'_1(\cdot)}(w^{-1})} \|\chi_j\|^{-1}_{L^{q'_2(\cdot)}(w^{-1})} \no\\
&\quad \les  2^{(k-j)w^+\alpha^+} 2^{k(\beta+n/s)} 2^{n\delta_2(j-k)} 2^{j(n-\beta)} 2^{-jn/s} \|\chi_j\|^{-1}_{L^{q_2(\cdot)}(w)} \|\chi_j\|^{-1}_{L^{q'_2(\cdot)}(w^{-1})} \no\\
&\quad \les 2^{\varepsilon_1(k-j)},
\end{align}
where $\varepsilon_1=w^+\alpha^+ + \beta+n/s-n\delta_2$. Therefore, by (\ref{ow-5})-(\ref{ow-8}), we have
\begin{align*}
&\|v(B_k)^{\alpha(\cdot)/n} \mathcal{H}^{b,m}_{\Omega,\beta} (f) \chi_k \|_{L^{q_2(\cdot)}(w)} 
\les \|b\|_\ast^m  \sum_{j=-\infty}^k (k-j)^m 
 \|v(B_j)^{\alpha(\cdot)/n} f_j\|_{L^{q_1(\cdot)}(w)} 2^{\varepsilon_1(k-j)} .
 \end{align*}
By Lemma \ref{tL4}, we have
\begin{align*}
\|\mathcal{H}^{b,m}_{\Omega,\beta} f\|_{M\dot{K}_{q_2(\cdot),\lambda}^{\alpha (\cdot),q),\theta}(v,w)}
& \les \sup_{\delta>0} \sup_{k_0 \in \mathbb{Z}} 2^{-k_0\lambda} \bigg( \delta^\theta \sum_{k=-\infty}^{-1} \|v(B_k)^{\alpha(\cdot)/n} \mathcal{H}^{b,m}_{\Omega,\beta} f \chi_k \|_{L^{q_2(\cdot)}(w)}^{p(1+\delta)} \bigg)^{\frac{1}{p(1+\delta)}} \\
&  \quad + \sup_{\delta>0} \sup_{k_0 \in \mathbb{Z}} 2^{-k_0\lambda} \bigg( \delta^\theta \sum_{k=0}^{k_0} \|v(B_k)^{\alpha(\cdot)/n} \mathcal{H}^{b,m}_{\Omega,\beta} f \chi_k \|_{L^{q_2(\cdot)}(w)}^{p(1+\delta)} \bigg)^{\frac{1}{p(1+\delta)}} \\
&=: F_1+F_2.
\end{align*}
Estimate $F_1$. Since $\varepsilon_1<0$, by  H\"{o}lder's inequality, we obtain 
\begin{align*}
F_1 &\les \sup_{\delta>0} \sup_{k_0 \in \mathbb{Z}} 2^{-k_0\lambda} \bigg( \delta^\theta \sum_{k=-\infty}^{-1} \|v(B_k)^{\alpha(\cdot)/n} \mathcal{H}^{b,m}_{\Omega,\beta} f \chi_k \|_{L^{q_2(\cdot)}(w)}^{p(1+\delta)} \bigg)^{\frac{1}{p(1+\delta)}} \\
&\les \|b\|_\ast^m \sup_{\delta>0} \sup_{k_0 \in \mathbb{Z}} 2^{-k_0\lambda} \bigg( \delta^\theta \sum_{k=-\infty}^{-1} \bigg(   \sum_{j=-\infty}^k  2^{\varepsilon_1(k-j)p(1+\delta)/2} \|v(B_j)^{\alpha(\cdot)/n} f_j\|_{L^{q_1(\cdot)}(w)}^{p(1+\delta)} \bigg) \\
& \quad \times  \bigg(   \sum_{j=-\infty}^k (k-j)^{m(p(1+\delta))^{\prime}} 2^{\varepsilon_1(k-j) (p(1+\delta))^{\prime}/2} \bigg)^{\frac{p(1+\delta)}{(p(1+\delta))^{\prime}}} \bigg)^{\frac{1}{p(1+\delta)}} \\
& \les \|b\|_\ast^m \sup_{\delta>0} \sup_{k_0 \in \mathbb{Z}} 2^{-k_0\lambda} \bigg( \delta^\theta \sum_{j=-\infty}^{-1} \|v(B_j)^{\alpha(\cdot)/n} f_j\|_{L^{q_1(\cdot)}(w)} ^{p(1+\delta)} \bigg)^{\frac{1}{p(1+\delta)}} \\
&\les \|b\|_\ast^m \| f\|_{M\dot{K}_{q_1(\cdot),\lambda}^{\alpha (\cdot),q),\theta}(v,w)}.
\end{align*}
Estimate $F_2$.
\begin{align*}
F_2 &\les \|b\|_\ast^m \sup_{\delta>0} \sup_{k_0 \in \mathbb{Z}} 2^{-k_0\lambda} \bigg( \delta^\theta \sum_{k=0}^{k_0} \bigg(   \sum_{j=-\infty}^k (k-j)^m 2^{\varepsilon_1(k-j)} \\
& \quad \times \|v(B_j)^{\alpha(\cdot)/n} f_j\|_{L^{q_1(\cdot)}(w)}  \bigg)^{p(1+\delta)} \bigg)^{\frac{1}{p(1+\delta)}} \\
& \les \|b\|_\ast^m \sup_{\delta>0} \sup_{k_0 \in \mathbb{Z}} 2^{-k_0\lambda} \bigg( \delta^\theta \sum_{k=0}^{k_0} \bigg(   \sum_{j=-\infty}^{-1} (k-j)^m 2^{\varepsilon_1(k-j)} \|v(B_j)^{\alpha(\cdot)/n} f_j\|_{L^{q_1(\cdot)}(w)}\\
& \quad + \sum_{j=0}^k (k-j)^m 2^{\varepsilon_1(k-j)} \|v(B_j)^{\alpha(\cdot)/n} f_j\|_{L^{q_1(\cdot)}(w)} \bigg)^{p(1+\delta)} \bigg)^{\frac{1}{p(1+\delta)}} \\
& \les \|b\|_\ast^m \sup_{\delta>0} \sup_{k_0 \in \mathbb{Z}} 2^{-k_0\lambda} \bigg( \delta^\theta \sum_{k=0}^{k_0} \bigg(   \sum_{j=-\infty}^{-1} (k-j)^m 2^{\varepsilon_1(k-j)} \|v(B_j)^{\alpha(\cdot)/n} f_j\|_{L^{q_1(\cdot)}(w)} \bigg)^{p(1+\delta)} \bigg)^{\frac{1}{p(1+\delta)}} \\
& \quad + \|b\|_\ast^m \sup_{\delta>0} \sup_{k_0 \in \mathbb{Z}} 2^{-k_0\lambda} \bigg( \delta^\theta \sum_{k=0}^{k_0} \sum_{j=0}^k (k-j)^m 2^{\varepsilon_1(k-j)} \|v(B_j)^{\alpha(\cdot)/n} f_j\|_{L^{q_1(\cdot)}(w)} \bigg)^{p(1+\delta)} \bigg)^{\frac{1}{p(1+\delta)}} \\
&=: F_{2,1} + F_{2,2}.
\end{align*}
Estimate $F_{2,1}$.  Since $\varepsilon_1<0$, by  H\"{o}lder's inequality, we have
\begin{align*}
F_{2,1} & \les \|b\|_\ast^m \sup_{\delta>0} \sup_{k_0 \in \mathbb{Z}} 2^{-k_0\lambda} \bigg( \delta^\theta \sum_{k=0}^{k_0} 2^{\varepsilon_1 kp(1+\delta)}   \\
& \quad \times \bigg(   \sum_{j=-\infty}^{-1} (k-j)^m 2^{-\varepsilon_1 j}  \|v(B_j)^{\alpha(\cdot)/n} f_j\|_{L^{q_1(\cdot)}(w)} \bigg)^{p(1+\delta)} \bigg)^{\frac{1}{p(1+\delta)}} \\
&\les \|b\|_\ast^m \sup_{\delta>0} \sup_{k_0 \in \mathbb{Z}} 2^{-k_0\lambda}  \bigg( \delta^\theta  \sum_{k=0}^{k_0} \bigg( \sum_{j=-\infty}^{-1} \|v(B_j)^{\alpha(\cdot)/n} f_j\|_{L^{q_1(\cdot)}(w)}^{p(1+\delta)} \bigg) \\
& \quad \times \bigg(   \sum_{j=-\infty}^{-1} (k-j)^{m(p(1+\delta))^{\prime}} 2^{-\varepsilon_1 j (p(1+\delta))^{\prime}}   \bigg)^{\frac{p(1+\delta)}{(p(1+\delta))^{\prime}}} \bigg)^{\frac{1}{p(1+\delta)}}  \\
& \les \|b\|_\ast^m \sup_{\delta>0} \sup_{k_0 \in \mathbb{Z}} 2^{-k_0\lambda} \bigg( \delta^\theta \sum_{j=-\infty}^{-1}  \|v(B_j)^{\alpha(\cdot)/n} f_j\|_{L^{q_1(\cdot)}(w)}^{p(1+\delta)} \bigg)^{\frac{1}{p(1+\delta)}}  \\
&\les \|b\|_\ast^m \| f\|_{M\dot{K}_{q_1(\cdot),\lambda}^{\alpha (\cdot),q),\theta}(v,w)} .
\end{align*}
Estimate $F_{2,2}$ is similar to estimate $F_1$, we have
\[ F_{2,2} \les \|b\|_\ast^m \| f\|_{M\dot{K}_{q_1(\cdot),\lambda}^{\alpha (\cdot),p),\theta}(v,w)}. \]
Combining the estimates of $E$ and $F$, we get
\[\|\mathcal{H}^{b,m}_{\Omega,\beta} f\|_{M\dot{K}_{q_2(\cdot),\lambda}^{\alpha (\cdot),q),\theta}(v,w)}  \les \|b\|_\ast^m \| f\|_{M\dot{K}_{q_1(\cdot),\lambda}^{\alpha (\cdot),p),\theta}(v,w)}. \]
This completes the proof.
\end{proof}

\begin{thm}\label{bw-T2}
Let $0<\beta<n/q^+_1$,  $\theta \in (0,\infty)$, $\lambda \in [0, \infty)$, $b \in {\rm BMO}(\rn)$, $\alpha(\cdot) \in L^\infty(\rn)$,  $1<p<\infty$,  $q_2(\cdot)\in C^{\log}(\rn) \cap \cp(\rn)$ and $q_1(\cdot)$ satisfying $1/q_2(\cdot)=1/q_1(\cdot) - \beta/n$.    
Let $v\in A_r$ for some $r \in [1,\infty)$, $w \in A_{q_2(\cdot)}$, $\Omega \in L^s(S^{n-1})$ with $s>q^{\prime}_1(\cdot)$ and $w^-\alpha^- -\beta+n\delta_1>0$,  where  $\delta_1 \in (0,1)$ is the constant in Lemma \ref{tt-L5} for $q(\cdot)$. 
Suppose that  $\alpha(\cdot)$ is log-H\"older continuous at infinity and at the origin. Then
\[\|\mathcal{H}^{\ast,b,m}_{\Omega,\beta} f\|_{M\dot{K}_{q_2(\cdot),\lambda}^{\alpha (\cdot),p),\theta}(v,w)} \les \|b\|_\ast^m \| f\|_{M\dot{K}_{q_1(\cdot),\lambda}^{\alpha (\cdot),p),\theta}(v,w)}.\]
\end{thm}

\begin{proof}
 Since bounded functions with compact support are dense in $M\dot{K}_{q_2(\cdot),\lambda}^{\alpha (\cdot),p),\theta}(v,w)$, we only consider that $f$ is a bounded functions with compact support and write
\[f=\sum^\infty_{j=-\infty}f \chi_j=:\sum^\infty_{j=-\infty}f_j.\]
By  Lemma \ref{tL2}, we have
\begin{align*}
&|v(B_k)^{\alpha(x)/n} \mathcal{H}^{\ast,b,m}_{\Omega,\beta} f(x) \chi_k(x) |\\
&\leq  |v(B_k)^{\alpha(x)/n}|  \int_{|t| \geq |x|} \frac{|b(x)-b(t)|^m |\Omega(x-t) f(t)| }{|t|^{n-\beta}} {\rm d} t \chi_k(x)\\
&\leq  v(B_k)^{\alpha(x)/n}  \int_{\rn \setminus B_k} \frac{|b(x)-b(t)|^m |\Omega(x-t) f(t)| }{|t|^{n-\beta}} {\rm d} t \chi_k(x)\\
&\leq  v(B_k)^{\alpha(x)/n} \sum_{j=k+1}^\infty \int_{D_j} \frac{|b(x)-b(t)|^m |\Omega(x-t) f(t)| }{|t|^{n-\beta}} {\rm d} t \chi_k(x)\\
& \leq  \sum_{j=k+1}^\infty 2^{(k-j)w^-\alpha^-} 2^{j(\beta-n)}  \int_{D_j}  |b(x)-b(t)|^m |v(B_j)^{\alpha(t)/n} \Omega(x-t) f(t)| {\rm d} t \chi_k(x) \\
& \les   \sum_{j=k+1}^\infty 2^{(k-j)w^-\alpha^-} 2^{j(\beta-n)} \int_{D_j}  |b(x)-b_{D_j}|^m |v(B_j)^{\alpha(t)/n} \Omega(x-t) f(t)| {\rm d} t \chi_k(x) \\
& \quad +   \sum_{j=k+1}^\infty 2^{(k-j)w^-\alpha^-} 2^{j(\beta-n)}  \int_{D_j}  |b(t)-b_{D_j}|^m |v(B_j)^{\alpha(t)/n} \Omega(x-t) f(t)| {\rm d} t \chi_k(x) \\
& =: G_1 + G_2.
\end{align*}

Estimate $G_1$. By  H\"{o}lder's inequality, we have
\begin{align}\label{ow-10}
G_1 &\les   \sum_{j=k+1}^\infty |b(x)-b_{D_j}|^m 2^{(k-j)w^-\alpha^-} 2^{j(\beta-n)} \no \\
& \quad \times \|v(B_j)^{\alpha(\cdot)/n} f_j\|_{L^{q_1(\cdot)}(w)} \|\Omega(x-\cdot)\chi_j\|_{L^{q'_1(\cdot)}(w^{-1})}  \chi_k(x) .
\end{align}
Taking the $L^{q_2(\cdot)}(w)$-norm on both sides of (\ref{ow-10}), by Lemma \ref{L10}, we have 
\begin{align}\label{ow-11}
\|G_1 \|_{L^{q_2(\cdot)}(w)}
& \les    \sum_{j=k+1}^\infty 2^{(k-j)w^-\alpha^-} 2^{j(\beta-n)} \| |b(x)-b_{D_j}|^m \chi_k\|_{L^{q_2(\cdot)}(w)}  \no\\
& \quad \times \|v(B_j)^{\alpha(\cdot)/n} f_j\|_{L^{q_1(\cdot)}(w)}  \|\Omega(x-\cdot)\chi_j\|_{L^{q'_1(\cdot)}(w^{-1})}  \no \\
& \les \|b\|^m_\ast   \sum_{j=k+1}^\infty (k-j)^m 2^{(k-j)w^-\alpha^-} 2^{j(\beta-n)} \no \\
& \quad \times \|v(B_j)^{\alpha(\cdot)/n} f_j\|_{L^{q_1(\cdot)}(w)}  \|\Omega(x-\cdot)\chi_j\|_{L^{q'_1(\cdot)}(w^{-1})} \| \chi_k\|_{L^{q_2(\cdot)}(w)}.
\end{align}

Estimate $G_2$. by H\"{o}lder's inequality, we have
\begin{align}\label{ow-12}
G_2 &\les   \sum_{j=k+1}^\infty 2^{(k-j)w^-\alpha^-} 2^{j(\beta-n)} \|v(B_j)^{\alpha(\cdot)/n} f_j\|_{L^{q_1(\cdot)}(w)} \no \\
& \quad \times  \| |b(x)-b_{D_j}|^m \Omega(x-\cdot)\chi_j\|_{L^{q'_1(\cdot)}(w^{-1})}  \chi_k(x) .
\end{align}
Taking the $L^{q_2(\cdot)}(w)$-norm on both sides of (\ref{ow-12}), by Lemma \ref{L10},we have
\begin{align}\label{ow-13}
\|G_2 \|_{L^{q_2(\cdot)}(w)} &\les   \sum_{j=k+1}^\infty 2^{(k-j)w^-\alpha^-} 2^{j(\beta-n)} \|v(B_j)^{\alpha(\cdot)/n} f_j\|_{L^{q_1(\cdot)}(w)} \no \\
& \quad \times  \| |b(x)-b_{D_j}|^m \Omega(x-\cdot)\chi_j\|_{L^{q'_1(\cdot)}(w^{-1})}  \|\chi_k(x) \|_{L^{q_2(\cdot)}(w)} \no\\
& \les \|b\|^m_\ast   \sum_{j=k+1}^\infty  2^{(k-j)w^-\alpha^-} 2^{j(\beta-n)} \|v(B_j)^{\alpha(\cdot)/n} f_j\|_{L^{q_1(\cdot)}(w)} \no\\
& \quad \times  \| \Omega(x-\cdot)\chi_j\|_{L^{q'_1(\cdot)}(w^{-1})}  \|\chi_k(x) \|_{L^{q_2(\cdot)}(w)}.
\end{align}
Thus, by (\ref{ow-11}) and (\ref{ow-13}), we obtain
\begin{align}\label{ow-14}
&\|v(B_k)^{\alpha(\cdot)/n} \mathcal{H}^{\ast,b,m}_{\Omega,\beta} (f) \chi_k \|_{L^{q_2(\cdot)}(w)} \no \\
&\quad  \les  \|b\|_\ast^m  \sum_{j=k+1}^\infty (k-j)^m  2^{(k-j)w^-\alpha^-} 2^{j(\beta-n)}  \no \\
& \quad \quad \times \|v(B_j)^{\alpha(\cdot)/n} f_j\|_{L^{q_1(\cdot)}(w)} \| \Omega(x-\cdot)\chi_j\|_{L^{q'_1(\cdot)}(w^{-1})}  \|\chi_k(x) \|_{L^{q_2(\cdot)}(w)}.
\end{align}
Note that $\|\chi_k\|_{L^{q_2(\cdot)}(w)} \leq \|\chi_{B_k}\|_{L^{q_2(\cdot)}(w)}$ and $\chi_{B_k} \les 2^{-k\beta} I_{\beta}(\chi_{B_k})$.  Thus, by Lemma \ref{tt-L3}, we have
\begin{align}\label{ow-15}
\|\chi_k\|_{L^{q_2(\cdot)}(w)}
& \les 2^{-k\beta} \|I_{\beta}(\chi_{B_k})\|_{L^{q_2(\cdot)}(w)} \les 2^{-k\beta} \|\chi_{B_k}\|_{L^{q_1(\cdot)}(w)} \no \\
&\les 2^{-k\beta} \|\chi_k\|_{L^{q_1(\cdot)}(w)}  .
\end{align}
Since $s>q_1^{\prime}(\cdot)$, there exists $\gamma^{\prime}(\cdot)$ such that $1/q^{\prime}_1(\cdot)=1/s + 1/\gamma^{\prime}(\cdot)$. Then by  H\"{o}lder's inequality and Lemma \ref{tt-L2}, we obtain
\begin{equation}\label{ow-18}
\|\Omega(x-\cdot)\chi_j\|_{L^{q'_1(\cdot)}(w^{-1})} \les \|\Omega(x-\cdot)\|_{L^s} \|\chi_j\|_{L^{\gamma'(\cdot)}(w^{-1})} \les 2^{jn/s} \|\chi_j\|_{L^{\gamma'(\cdot)}(w^{-1})}.
\end{equation}
By Lemmas \ref{tt-L4} and \ref{tt-L5}, Definition \ref{dg1}, (\ref{ow-18})  and (\ref{ow-15}), we have
\begin{align}\label{ow-16}
& 2^{(k-j)w^-\alpha^-} 2^{j(\beta-n)}  \|\chi_j\|_{L^{\gamma'(\cdot)}(w^{-1})} \|\chi_k\|_{L^{q_2(\cdot)}(w)} \no\\
& \quad \les 2^{(k-j)w^-\alpha^-} 2^{j(\beta-n)}   \|\chi_j\|_{L^{q'_1(\cdot)}(w^{-1})} \|\chi_k\|_{L^{q_2(\cdot)}(w)} \no\\
& \quad \les 2^{(k-j)w^-\alpha^-} 2^{j\beta}  \|\chi_j\|_{L^{q_1(\cdot)}(w)}^{-1} \|\chi_k\|_{L^{q_2(\cdot)}(w)} \no\\
& \quad \les 2^{(k-j)w^-\alpha^-} 2^{j\beta}  2^{n\delta_1(k-j)}  \|\chi_k\|_{L^{q_1(\cdot)}(w)}^{-1} \|\chi_k\|_{L^{q_2(\cdot)}(w)} \no\\
& \quad \les 2^{(k-j)w^-\alpha^-} 2^{j\beta}  2^{n\delta_1(k-j)} 2^{-k\beta} \|\chi_k\|_{L^{q_1(\cdot)}(w)}^{-1} \|\chi_k\|_{L^{q_1(\cdot)}(w)} \no\\
& \quad \les 2^{(k-j)\varepsilon_2},
\end{align}
where $\varepsilon_2=w^-\alpha^- -\beta+n\delta_1$.
Thus, by (\ref{ow-14}), (\ref{ow-6}) (\ref{ow-15}) and (\ref{ow-16}), we obtain
\begin{equation*}
\|v(B_k)^{\alpha(\cdot)/n} \mathcal{H}^{\ast,b}_{\Omega,\beta} (f) \chi_k \|_{L^{q_2(\cdot)}(w)} \les \|b\|^m_\ast  \sum_{j=k+1}^\infty (k-j)^m \|v(B_j)^{\alpha(\cdot)/n} f_j\|_{L^{q_1(\cdot)}(w)} 2^{\varepsilon_2(k-j)} .
\end{equation*}
By Lemma \ref{tL4}, we have
\begin{align*}
&\|\mathcal{H}^{\ast,b,m}_{\Omega,\beta} f\|_{M\dot{K}_{q_2(\cdot),\lambda}^{\alpha (\cdot),p),\theta}(v,w)}\\
& \quad \les \sup_{\delta>0} \sup_{k_0 \in \mathbb{Z}} 2^{-k_0\lambda} \bigg( \delta^\theta \sum_{k=-\infty}^{-1} \|v(B_k)^{\alpha(\cdot)/n} \mathcal{H}^{\ast,b,m}_{\Omega,\beta} f \chi_k \|_{L^{q_2(\cdot)}(w)}^{p(1+\delta)} \bigg)^{\frac{1}{p(1+\delta)}} \\
&  \quad \quad + \sup_{\delta>0} \sup_{k_0 \in \mathbb{Z}} 2^{-k_0\lambda} \bigg( \delta^\theta \sum_{k=0}^{k_0} \|v(B_k)^{\alpha(\cdot)/n} \mathcal{H}^{\ast,b,m}_{\Omega,\beta} f \chi_k \|_{L^{q_2(\cdot)}(w)}^{p(1+\delta)} \bigg)^{\frac{1}{p(1+\delta)}} \\
&=: H_1+H_2.
\end{align*}
Estimate $H_1$.
 \begin{align*}
H_1 &\les \|b\|_\ast^m \sup_{\delta>0} \sup_{k_0 \in \mathbb{Z}} 2^{-k_0\lambda} \bigg( \delta^\theta \sum_{k=-\infty}^{-1} \bigg(  \sum_{j=k+1}^{-1} \|v(B_j)^{\alpha(\cdot)/n} f_j\|_{L^{q_1(\cdot)}(w)}\\
& \quad \times  (k-j)^m 2^{\varepsilon_2(k-j)} \bigg)^{p(1+\delta)} \bigg)^{\frac{1}{p(1+\delta)}} \\
& \quad + \|b\|_\ast^m \sup_{\delta>0} \sup_{k_0 \in \mathbb{Z}} 2^{-k_0\lambda} \bigg( \delta^\theta \sum_{k=-\infty}^{-1} \bigg(  \sum_{j=0}^\infty \|v(B_j)^{\alpha(\cdot)/n} f_j\|_{L^{q_1(\cdot)}(w)} \\
& \quad \quad \times (k-j)^m 2^{\varepsilon_2(k-j)} \bigg)^{p(1+\delta)} \bigg)^{\frac{1}{p(1+\delta)}}\\
&=: H_{1,1}+H_{1,2}.
\end{align*}

Estimate $H_{1,1}$. Since $\varepsilon_2>0$, by H\"{o}lder's inequality, we have
 \begin{align*}
H_{1,1} &  \les \|b\|_\ast^m \sup_{\delta>0} \sup_{k_0 \in \mathbb{Z}} 2^{-k_0\lambda} \bigg( \delta^\theta \sum_{k=-\infty}^{-1} \bigg(  \sum_{j=k+1}^{-1} 2^{\varepsilon_2(k-j)p(1+\delta)/2} \|v(B_j)^{\alpha(\cdot)/n} f_j\|_{L^{q_1(\cdot)}(w)}^{p(1+\delta)} \bigg) \\
& \quad \times   \bigg( \sum_{j=k+1}^{-1} (k-j)^{m(p(1+\delta))^{\prime}} 2^{\varepsilon_2(k-j)(p(1+\delta))^{\prime}/2} \bigg)^{\frac{p(1+\delta)}{(p(1+\delta))^{\prime}}} \bigg)^{\frac{1}{p(1+\delta)}}\\
& \les   \|b\|_\ast^m \sup_{\delta>0} \sup_{k_0 \in \mathbb{Z}} 2^{-k_0\lambda}  \bigg( \sum_{j=-\infty}^{-1} \|v(B_j)^{\alpha(\cdot)/n} f_j\|_{L^{q_1(\cdot)}(w)}^{p(1+\delta)} \bigg)^{\frac{1}{p(1+\delta)}} \\
& \les \|b\|_\ast^m \| f\|_{M\dot{K}_{q_1(\cdot),\lambda}^{\alpha (\cdot),p),\theta}(v,w)}.
\end{align*}

Estimate $H_{1,2}$. Since $\varepsilon_2>0$, by  H\"{o}lder's inequality, we have
 \begin{align*}
H_{1,2} & \les \|b\|_\ast^m \sup_{\delta>0} \sup_{k_0 \in \mathbb{Z}} 2^{-k_0\lambda} \bigg( \delta^\theta \sum_{k=-\infty}^{-1} 2^{\varepsilon_2 kp(1+\delta)}   \\
& \quad \times \bigg(   \sum_{j=0}^\infty (k-j)^m 2^{-\varepsilon_1 j}  \|v(B_j)^{\alpha(\cdot)/n} f_j\|_{L^{q_1(\cdot)}(w)} \bigg)^{p(1+\delta)} \bigg)^{\frac{1}{p(1+\delta)}} \\
&\les \|b\|_\ast^m \sup_{\delta>0} \sup_{k_0 \in \mathbb{Z}} 2^{-k_0\lambda}  \bigg( \delta^\theta  \sum_{k=-\infty}^{-1} \bigg( \sum_{j=0}^\infty \|v(B_j)^{\alpha(\cdot)/n} f_j\|_{L^{q_1(\cdot)}(w)}^{p(1+\delta)} \bigg) \\
& \quad \times \bigg(   \sum_{j=0}^\infty (k-j)^{m(p(1+\delta))^{\prime}} 2^{-\varepsilon_1 j (p(1+\delta))^{\prime}}   \bigg)^{\frac{p(1+\delta)}{(p(1+\delta))^{\prime}}} \bigg)^{\frac{1}{p(1+\delta)}}  \\
& \les \|b\|_\ast^m \sup_{\delta>0} \sup_{k_0 \in \mathbb{Z}} 2^{-k_0\lambda} \bigg( \delta^\theta \sum_{j=-\infty}^{-1}  \|v(B_j)^{\alpha(\cdot)/n} f_j\|_{L^{q_1(\cdot)}(w)}^{p(1+\delta)} \bigg)^{\frac{1}{p(1+\delta)}}  \\
&\les \|b\|_\ast^m \| f\|_{M\dot{K}_{q_1(\cdot),\lambda}^{\alpha (\cdot),p),\theta}(v,w)} .
\end{align*}

Since the estimatation of  $H_2$ is similar to that of $H_{1,1}$, we have
\[ H_2 \les \|b\|_\ast^m \| f\|_{M\dot{K}_{q_1(\cdot),\lambda}^{\alpha (\cdot),p),\theta}(v,w)}.\]
Combining the estimates of $G$ and $H$, we get
\[\|\mathcal{H}^{\ast,b,m}_{\Omega,\beta} f\|_{M\dot{K}_{q_2(\cdot),\lambda}^{\alpha (\cdot),p),\theta}(v,w)} \les \|b\|_\ast^m \| f\|_{M\dot{K}_{q_1(\cdot),\lambda}^{\alpha (\cdot),p),\theta}(v,w)}.\]
This completes the proof.
\end{proof}

\section{Commutators by the Hardy operator and Lipschitz functions}\label{wbl-s-3}
In this section, we show that the boundedness of commutators generated by the $n$-dimensional fractional Hardy operator with rough kernel and its adjoint operator and Lipschitz functions  on two weighted grand Herz-Morrey spaces with variable exponents. 

Let $0 < \sigma \leq 1$ and $f \in L^{1}_{\rm loc}(\rn)$. The Lipschitz space ${\rm Lip}^\sigma(\rn)$ is defined as the set of all functions $f : \rn \rightarrow \cc$ such that
\[ \|f\|_{{\rm Lip}^\sigma} := \sup_{x,y \in \rn,x\neq y} \frac{|f(x)-f(y)|}{|x-y|^\sigma}.\]

\begin{thm}
Let $0<\beta<n/q^+_1$,  $\theta \in (0,\infty)$, $\lambda \in [0, \infty)$, $0<\sigma<1$, $b \in {\rm Lip}^\sigma(\rn)$, $\alpha(\cdot) \in L^\infty(\rn)$,  $1<p<\infty$,  $q_2(\cdot)\in C^{\log}(\rn) \cap \cp(\rn)$ and $q_1(\cdot)$ satisfying $1/q_2(\cdot)=1/q_1(\cdot) - \beta/n$.    
Let $v\in A_r$ for some $r \in [1,\infty)$, $w \in A_{q_2(\cdot)}$, $\Omega \in L^s(S^{n-1})$ with $s>q^{\prime}_1(\cdot)$ and $w^+\alpha^+ + \beta+m\sigma +n/s-n\delta_2<0$,  where  $\delta_2 \in (0,1)$ is the constant in Lemma \ref{tt-L5} for $q(\cdot)$. 
Suppose that  $\alpha(\cdot)$ is log-H\"older continuous at infinity and at the origin. Then
\[\|\mathcal{H}^{b,m}_{\Omega,\beta} f\|_{M\dot{K}_{q_2(\cdot),\lambda}^{\alpha (\cdot),p),\theta}(v,w)} \les  \|b\|^m_{{\rm Lip}^\sigma} \| f\|_{M\dot{K}_{q_1(\cdot),\lambda}^{\alpha (\cdot),p),\theta}(v,w)}.\]
\end{thm}

\begin{proof}
 Since bounded functions with compact support are dense in $M\dot{K}_{q_2(\cdot),\lambda}^{\alpha (\cdot),p),\theta}(v,w)$, we only consider that $f$ is a bounded functions with compact support and write
\[f=\sum^\infty_{j=-\infty}f \chi_j=:\sum^\infty_{j=-\infty}f_j.\]
By  H\"{o}lder's inequality, Lemmas \ref{tL2} and \ref{tt-L2}, we have
\begin{align*}
&|v(B_k)^{\alpha(x)/n} \mathcal{H}^{b,m}_{\Omega,\beta} f(x) \chi_k(x)| \\
&\leq  v(B_k)^{\alpha(x)/n} \frac{1}{|x|^{n-\beta}} \int_{B_k} |b(x)-b(t)|^m |\Omega(x-t) f(t)| {\rm d} t \chi_k(x)\\
&\leq  v(B_k)^{\alpha(x)/n} \frac{1}{|x|^{n-\beta}} \sum_{j=-\infty}^k \int_{D_j} |b(x)-b(t)|^m |\Omega(x-t) f(t)| {\rm d} t \chi_k(x)\\
& \leq 2^{(k-j)w^+\alpha^+} 2^{k(\beta+m\sigma-n)}  \|b\|^m_{{\rm Lip}^\sigma}  \sum_{j=-\infty}^k \int_{D_j} |v(B_j)^{\alpha(t)/n} \Omega(x-t) f(t)| {\rm d} t \chi_k(x)\\
& \leq 2^{(k-j)w^+\alpha^+} 2^{k(\beta+m\sigma-n)} \|b\|^m_{{\rm Lip}^\sigma}  \sum_{j=-\infty}^k \|v(B_j)^{\alpha(\cdot)/n} f_j\|_{L^{q_1(\cdot)}(w)} \|\Omega(x-\cdot)\chi_j\|_{L^{q'_1(\cdot)}(w^{-1})}  \chi_k(x)
\end{align*}
Taking the $L^{q_2(\cdot)}(w)$-norm on both sides of the above inequality, we have
\begin{align*}
\|v(B_k)^{\alpha(\cdot)/n} \mathcal{H}^{b,m}_{\Omega,\beta} (f) \chi_k \|_{L^{q_2(\cdot)}(w)} 
& \les 2^{(k-j)w^+\alpha^+} 2^{k(\beta+m\sigma-n)} \|b\|^m_{{\rm Lip}^\sigma}  \sum_{j=-\infty}^k \|\chi_k\|_{L^{q_2(\cdot)}(w)} \\
&\quad \times  \|\Omega(x-\cdot)\chi_j\|_{L^{q'_1(\cdot)}(w^{-1})}  . 
\end{align*}
Note that $\|\chi_j\|_{L^{q_2(\cdot)}(w)} \leq \|\chi_{B_j}\|_{L^{q_2(\cdot)}(w)}$ and $\chi_{B_j} \les 2^{-j(\beta+m\sigma)} I_{\beta+m\sigma}(\chi_{B_j})$.  Thus, by Lemma \ref{tt-L3} and Definition \ref{dg1}, we obtain
\begin{align}\label{lw-1}
\|\chi_j\|_{L^{q_2(\cdot)}(w)}
& \les 2^{-j(\beta+m\sigma)} \|I_{\beta+m\sigma}(\chi_{B_j})\|_{L^{q_2(\cdot)}(w)} \les 2^{-j(\beta+m\sigma)} \|\chi_{B_j}\|_{L^{q_1(\cdot)}(w)} \no \\
&\les 2^{-j(\beta+m\sigma )} \|\chi_j\|_{L^{q_1(\cdot)}(w)}  \les 2^{j(n-\beta-m\sigma )} \|\chi_j\|^{-1}_{L^{q'_1(\cdot)}(w^{-1})}.
\end{align}
Since $s>q_1^{\prime}(\cdot)$, there exists $\gamma^{\prime}(\cdot)$ such that $1/q^{\prime}_1(\cdot)=1/s + 1/\gamma^{\prime}(\cdot)$. 
Then by (\ref{ow-6}), Lemma \ref{tt-L4},  Definition \ref{dg1}, Lemma \ref{tt-L5}  and (\ref{lw-1}), we have
\begin{align}\label{lw-2}
& 2^{(k-j)w^+\alpha^+} 2^{k(\beta+m\sigma -n)} \|\Omega(x-\cdot)\chi_j\|_{L^{q'_1(\cdot)}(w^{-1})} \|\chi_k\|_{L^{q_2(\cdot)}(w)} \no\\
&\quad  \les 2^{(k-j)w^+\alpha^+} 2^{k(\beta+m\sigma -n+n/s)} \|\chi_j\|_{L^{\gamma'(\cdot)}(w^{-1})} \|\chi_k\|_{L^{q_2(\cdot)}(w)} \no\\
& \quad \les 2^{(k-j)w^+\alpha^+} 2^{k(\beta+m\sigma -n+n/s)} 2^{-jn/s} \|\chi_j\|_{L^{q'_1(\cdot)}(w^{-1})} \|\chi_k\|_{L^{q_2(\cdot)}(w)} \no\\
&\quad \les 2^{(k-j)w^+\alpha^+} 2^{k(\beta+m\sigma +n/s)} 2^{-jn/s} \|\chi_j\|_{L^{q'_1(\cdot)}(w^{-1})} \|\chi_k\|^{-1}_{L^{q'_2(\cdot)}(w^{-1})} \no\\
&\quad \les  2^{(k-j)w^+\alpha^+} 2^{k(\beta+m\sigma +n/s)} 2^{n\delta_2(j-k)} 2^{-jn/s} \|\chi_j\|_{L^{q'_1(\cdot)}(w^{-1})} \|\chi_j\|^{-1}_{L^{q'_2(\cdot)}(w^{-1})} \no\\
&\quad \les  2^{(k-j)w^+\alpha^+} 2^{k(\beta+m\sigma +n/s)} 2^{n\delta_2(j-k)} 2^{j(n-\beta-m\sigma )} 2^{-jn/s} \|\chi_j\|^{-1}_{L^{q_2(\cdot)}(w)} \|\chi_j\|^{-1}_{L^{q'_2(\cdot)}(w^{-1})} \no\\
&\quad \les 2^{\varepsilon_3(k-j)},
\end{align}
where $\varepsilon_3=w^+\alpha^+ + \beta+m\sigma +n/s-n\delta_2$.
Thus, by (\ref{lw-1}) and (\ref{lw-2}), we get
\[\|v(B_k)^{\alpha(\cdot)/n} \mathcal{H}^{b,m}_{\Omega,\beta} (f) \chi_k \|_{L^{q_2(\cdot)}(w)} \les \|b\|^m_{{\rm Lip}^\sigma}  \sum_{j=-\infty}^k \|v(B_j)^{\alpha(\cdot)/n} f_j\|_{L^{q_1(\cdot)}(w)} 2^{\varepsilon_3(k-j)} .\]
Therefore, by Lemma \ref{tL4}, we have
\begin{align*}
\|\mathcal{H}^{b,m}_{\Omega,\beta} f\|_{M\dot{K}_{q_2(\cdot),\lambda}^{\alpha (\cdot),p),\theta}(v,w)}
& \les \sup_{\delta>0} \sup_{k_0 \in \mathbb{Z}} 2^{-k_0\lambda} \bigg( \delta^\theta \sum_{k=-\infty}^{-1} \|v(B_k)^{\alpha(\cdot)/n} \mathcal{H}^{b,m}_{\Omega,\beta} f \chi_k \|_{L^{q_2(\cdot)}(w)}^{p(1+\delta)} \bigg)^{\frac{1}{p(1+\delta)}} \\
&  \quad + \sup_{\delta>0} \sup_{k_0 \in \mathbb{Z}} 2^{-k_0\lambda} \bigg( \delta^\theta \sum_{k=0}^{k_0} \|v(B_k)^{\alpha(\cdot)/n} \mathcal{H}^{b,m}_{\Omega,\beta} f \chi_k \|_{L^{q_2(\cdot)}(w)}^{p(1+\delta)} \bigg)^{\frac{1}{p(1+\delta)}} \\
&=: F_3+F_4.
\end{align*}
Estimating $F_3$ and $F_4$ is the same as estimating $F_1$ and $F_2$ in Theorem \ref{bw-T1}, respectively. Notice that we used the fact that $\varepsilon_3<0$. Thus, we obtain
\[\|\mathcal{H}^{b,m}_{\Omega,\beta} f\|_{M\dot{K}_{q_2(\cdot),\lambda}^{\alpha (\cdot),p),\theta}(v,w)} \les  \|b\|^m_{{\rm Lip}^\sigma} \| f\|_{M\dot{K}_{q_1(\cdot),\lambda}^{\alpha (\cdot),p),\theta}(v,w)}.\]
This completes the proof.
\end{proof}

\begin{thm}
Let $0<\beta<n/q^+_1$,  $\theta \in (0,\infty)$, $\lambda \in [0, \infty)$, $0<\sigma<1$, $b \in {\rm Lip}^\sigma(\rn)$, $\alpha(\cdot) \in L^\infty(\rn)$,  $1<p<\infty$,  $q_2(\cdot)\in C^{\log}(\rn) \cap \cp(\rn)$ and $q_1(\cdot)$ satisfying $1/q_2(\cdot)=1/q_1(\cdot) - \beta/n$.    
Let $v\in A_r$ for some $r \in [1,\infty)$, $w \in A_{q_2(\cdot)}$, $\Omega \in L^s(S^{n-1})$ with $s>q^{\prime}_1(\cdot)$ and $w^-\alpha^- -\beta-m\sigma +n\delta_1>0$,  where  $\delta_1 \in (0,1)$ is the constant in Lemma \ref{tt-L5} for $q(\cdot)$. 
Suppose that  $\alpha(\cdot)$ is log-H\"older continuous at infinity and at the origin. Then
\[\|\mathcal{H}^{\ast,b,m}_{\Omega,\beta} f\|_{M\dot{K}_{q_2(\cdot),\lambda}^{\alpha (\cdot),p),\theta}(v,w)} \les \|b\|^m_{{\rm Lip}^\sigma} \| f\|_{M\dot{K}_{q_1(\cdot),\lambda}^{\alpha (\cdot),p),\theta}(v,w)}.\]
\end{thm}

\begin{proof}
 Since bounded functions with compact support are dense in $M\dot{K}_{q_2(\cdot),\lambda}^{\alpha (\cdot),p),\theta}(v,w)$, we only consider that $f$ is a bounded functions with compact support and write
 \[f=\sum^\infty_{j=-\infty}f \chi_j=:\sum^\infty_{j=-\infty}f_j.\]
By H\"{o}lder's inequality, Lemma \ref{tL2}, we have
\begin{align*}
&|v(B_k)^{\alpha(x)/n} \mathcal{H}^{\ast,b,m}_{\Omega,\beta} f(x) \chi_k(x) |\\
&\leq  v(B_k)^{\alpha(x)/n}  \int_{|t| \geq |x|} \frac{|b(x)-b(t)|^m |\Omega(x-t) f(t)|}{|t|^{n-\beta}} {\rm d} t \chi_k(x)\\
&\leq  v(B_k)^{\alpha(x)/n}  \int_{\rn \setminus B_k} \frac{|b(x)-b(t)|^m |\Omega(x-t) f(t)|}{|t|^{n-\beta}} {\rm d} t \chi_k(x)\\
&\leq  v(B_k)^{\alpha(x)/n} \sum_{j=k+1}^\infty \int_{D_j} \frac{|b(x)-b(t)|^m |\Omega(x-t) f(t)|}{|t|^{n-\beta}} {\rm d} t \chi_k(x)\\
& \leq 2^{(k-j)w^-\alpha^-} 2^{j(\beta+m\sigma -n)}  \|b\|^m_{{\rm Lip}^\sigma} \sum_{j=k+1}^\infty \int_{D_j} |v(B_j)^{\alpha(t)/n} \Omega(x-t) f(t)| {\rm d} t \chi_k(x)\\
& \leq 2^{(k-j)w^-\alpha^-} 2^{j(\beta+m\sigma -n)} \|b\|^m_{{\rm Lip}^\sigma} \sum_{j=k+1}^\infty \|v(B_j)^{\alpha(\cdot)/n} f_j\|_{L^{q_1(\cdot)}(w)} \|\Omega(x-\cdot)\chi_j\|_{L^{q'_1(\cdot)}(w^{-1})}  \chi_k(x)
\end{align*}
Taking the $L^{q_2(\cdot)}(w)$-norm on both sides of the above inequality, we have
\begin{align*}
\|v(B_k)^{\alpha(\cdot)/n} \mathcal{H}^{\ast,b,m}_{\Omega,\beta} (f) \chi_k \|_{L^{q_2(\cdot)}(w)} 
& \les 2^{(k-j)w^-\alpha^-} 2^{j(\beta+m\sigma -n)} \|b\|^m_{{\rm Lip}^\sigma}  \sum_{j=-\infty}^k  \|\chi_k\|_{L^{q_2(\cdot)}(w)}  \\
&\quad  \times \|v(B_j)^{\alpha(\cdot)/n} f_j\|_{L^{q_1(\cdot)}(w)} \|\Omega(x-\cdot)\chi_j\|_{L^{q'_1(\cdot)}(w^{-1})}  .
\end{align*}
Note that $\|\chi_k\|_{L^{q_2(\cdot)}(w)} \leq \|\chi_{B_k}\|_{L^{q_2(\cdot)}(w)}$ and $\chi_{B_k} \les 2^{-k(\beta+m\sigma )} I_{\beta+m\sigma }(\chi_{B_k})$.  Thus, by Lemma \ref{tt-L3}, we have
\begin{align}\label{lw-4}
\|\chi_k\|_{L^{q_2(\cdot)}(w)}
& \les 2^{-k(\beta+m\sigma )} \|I_{\beta+m\sigma }(\chi_{B_k})\|_{L^{q_2(\cdot)}(w)} \les 2^{-k(\beta+m\sigma )} \|\chi_{B_k}\|_{L^{q_1(\cdot)}(w)} \no \\
&\les 2^{-k(\beta+m\sigma )} \|\chi_k\|_{L^{q_1(\cdot)}(w)}  .
\end{align}
By (\ref{ow-18}) Lemma \ref{tt-L4}, Definition \ref{dg1}, Lemma \ref{tt-L5}  and (\ref{lw-4}), we have
\begin{align}\label{lw-5}
& 2^{(k-j)w^-\alpha^-} 2^{j(\beta+m\sigma -n)}  \|\Omega(x-\cdot)\chi_j\|_{L^{q'_1(\cdot)}(w^{-1})}   \|\chi_k\|_{L^{q_2(\cdot)}(w)} \no\\
& \quad \les 2^{(k-j)w^-\alpha^-} 2^{j(\beta+m\sigma -n)} 2^{jn/s} \|\chi_j\|_{L^{\gamma'(\cdot)}(w^{-1})} \|\chi_k\|_{L^{q_2(\cdot)}(w)} \no\\
& \quad \les 2^{(k-j)w^-\alpha^-} 2^{j(\beta+m\sigma -n)}  \|\chi_j\|_{L^{q'_1(\cdot)}(w^{-1})} \|\chi_k\|_{L^{q_2(\cdot)}(w)} \no\\
& \quad \les 2^{(k-j)w^-\alpha^-} 2^{j(\beta+m\sigma )}   \|\chi_j\|_{L^{q_1(\cdot)}(w)}^{-1} \|\chi_k\|_{L^{q_2(\cdot)}(w)} \no\\
& \quad \les 2^{(k-j)w^-\alpha^-} 2^{j(\beta+m\sigma )}  2^{n\delta_1(k-j)}  \|\chi_k\|_{L^{q_1(\cdot)}(w)}^{-1} \|\chi_k\|_{L^{q_2(\cdot)}(w)} \no\\
& \quad \les 2^{(k-j)w^-\alpha^-} 2^{j(\beta+m\sigma )}  2^{n\delta_1(k-j)} 2^{-k(\beta+m\sigma )} \|\chi_k\|_{L^{q_1(\cdot)}(w)}^{-1} \|\chi_k\|_{L^{q_1(\cdot)}(w)} \no\\
& \quad \les 2^{(k-j)\varepsilon_4},
\end{align}
where $\varepsilon_4=w^-\alpha^- -\beta-m\sigma +n\delta_1$.
Thus, by (\ref{lw-5}), we obtain
\begin{equation*}
\|v(B_k)^{\alpha(\cdot)/n} \mathcal{H}^{\ast,b,m}_{\Omega,\beta} (f) \chi_k \|_{L^{q_2(\cdot)}(w)} \les \|b\|^m_{{\rm Lip}^\sigma}  \sum_{j=k+1}^\infty \|v(B_j)^{\alpha(\cdot)/n} f_j\|_{L^{q_1(\cdot)}(w)} 2^{\varepsilon_4(k-j)} .
\end{equation*}
Therefore, by Lemma \ref{tL4}, we have
\begin{align*}
\|\mathcal{H}^{\ast,b,m}_{\Omega,\beta} f\|_{M\dot{K}_{q_2(\cdot),\lambda}^{\alpha (\cdot),p),\theta}(v,w)}
&  \les \sup_{\delta>0} \sup_{k_0 \in \mathbb{Z}} 2^{-k_0\lambda} \bigg( \delta^\theta \sum_{k=-\infty}^{-1} \|v(B_k)^{\alpha(\cdot)/n} \mathcal{H}^{\ast,b}_{\Omega,\beta} f \chi_k \|_{L^{q_2(\cdot)}(w)}^{p(1+\delta)} \bigg)^{\frac{1}{p(1+\delta)}} \\
&  \quad + \sup_{\delta>0} \sup_{k_0 \in \mathbb{Z}} 2^{-k_0\lambda} \bigg( \delta^\theta \sum_{k=0}^{k_0} \|v(B_k)^{\alpha(\cdot)/n} \mathcal{H}^{\ast,b}_{\Omega,\beta} f \chi_k \|_{L^{q_2(\cdot)}(w)}^{p(1+\delta)} \bigg)^{\frac{1}{p(1+\delta)}} \\
&=: H_3+H_4.
\end{align*}
Estimating $H_3$ and $H_4$ is the same as estimating $H_1$ and $H_2$ in Theorem \ref{bw-T2}, respectively. Notice that we used the fact that $\varepsilon_4>0$. Thus, we obtain
\[\|\mathcal{H}^{\ast,b,m}_{\Omega,\beta} f\|_{M\dot{K}_{q_2(\cdot),\lambda}^{\alpha (\cdot),p),\theta}(v,w)} \les \|b\|^m_{{\rm Lip}^\sigma} \| f\|_{M\dot{K}_{q_1(\cdot),\lambda}^{\alpha (\cdot),p),\theta}(v,w)}.\]
This completes the proof.
\end{proof}

\end{document}